\documentclass[11pt]{article}
\usepackage{amssymb,amsmath,latexsym}
\usepackage{amsthm}
\usepackage[shortlabels]{enumitem}
\usepackage{caption}
\usepackage{xcolor}
\usepackage{float}
\usepackage{graphicx}

\usepackage{url}
\usepackage{breakurl}
\usepackage[colorlinks=true,citecolor=blue,linkcolor=blue,urlcolor=blue,bookmarks,bookmarksopen,bookmarksdepth=2,backref=page,breaklinks]{hyperref}

\usepackage{bookmark}
\bookmarksetup{
	numbered, 
	open,
}

\renewcommand*{\backref}[1]{}
\renewcommand*{\backrefalt}[4]{%
	\ifcase #1 (Not cited.)%
	\or        (Cited on page~#2.)%
	\else      (Cited on pages~#2.)%
	\fi}

\newcommand{\N}{\mathbb N}
\newcommand{\Z}{\mathbb Z}

\newcommand{\F}{\mathbb F}

\DeclareMathOperator{\GL}{GL}

\DeclareMathOperator{\Tr}{Tr}
\DeclareMathOperator{\wt}{wt}
\DeclareMathOperator{\Aut}{Aut}
\DeclareMathOperator{\ELM}{ELM}
\DeclareMathOperator{\EL}{EL}
\theoremstyle{plain}
\newtheorem{theorem}{Theorem}[section]
\newtheorem{lemma}[theorem]{Lemma}
\newtheorem{proposition}[theorem]{Proposition}

\theoremstyle{definition}
\newtheorem{conjecture}[theorem]{Conjecture}

\newtheorem{definition}[theorem]{Definition}

\numberwithin{theorem}{section}
\numberwithin{equation}{section}
\numberwithin{table}{section}
\numberwithin{figure}{section}
\DeclareMathOperator{\image}{Im}
\begin{document}
\title{On a recent extension of a family of biprojective APN functions}
\author{Lukas K\"olsch}

\author{
	Lukas K\"olsch\\
%	\small Otto von Guericke University Magdeburg\\[-0.8ex]
%	\small Universit\"{a}tsplatz 2, 39106, Magdeburg, Germany\\
%	\small\tt alexandr.polujan@$\{$gmail.com,ovgu.de$\}$, alexander.pott@ovgu.de
}

\author{Lukas K\"olsch \vspace{0.4cm} \ \\
University of South Florida \\\tt lukas.koelsch.math@gmail.com
}

\date{\today}
\maketitle
\abstract{
APN functions play a big role as primitives in symmetric cryptography as building blocks that yield optimal resistance to differential attacks.
In this note, we consider a recent extension of a biprojective APN family by G\"olo\u{g}lu defined on $\mathbb{F}_{2^{2m}}$. We show that this generalization yields functions equivalent to G\"olo\u{g}lu's original family if $3\nmid m$. If $3|m$ we show exactly how many inequivalent APN functions this new family contains. We also show that the family has the minimal image set size for an APN function and determine its Walsh spectrum, hereby settling some open problems. In our proofs, we leverage a group theoretic technique recently developed by G\"olo\u{g}lu and the author in conjunction with a group action on the set of projective polynomials.

%Value distribution of perfect nonlinear functions is an important characteristic which can tell  a lot about the function. It has mostly been studied for $p$-ary functions, planar functions and APN functions. In this paper, we develop further the theory of value distributions of perfect nonlinear cases, with a focus on the in-between cases as well as in the framework of non-elementary abelian groups. Introduce the notion of almost balanced,  generalize the known results on Nyberg as well as several results for the planar functions. \textbf{Still in progress.}\ \\[2mm]
}

\thispagestyle{empty}
\section{Introduction and preliminaries}
Let $ \F_{2^n}$ be the finite field with $2^n$ elements. Vectorial Boolean functions $F \colon \F_{2^n} \rightarrow \F_{2^n}$ play an important role in the construction of many block ciphers as potential building blocks for S-Boxes of substitution-permutation networks (SPNs). To be secure against differential attacks, the differential uniformity of an S-Box must be low, for a comprehensive overview, we refer the reader to~\cite[Section 3.4]{Carlet2021}:
\begin{definition}
   A function $F \colon \F_{2^n} \rightarrow \F_{2^n}$ has \textit{differential uniformity} $\delta$ if $$\delta=\max_{b \in \F_{2^n},a \in \F_{2^n}^*}|\{x\in\F_{2^n}\colon F(x+a)+F(x)=b\}|.$$
\end{definition}
Clearly, $x \in \F_{2^n}$ satisfies the equation $F(x+a)+F(x)=b$ if and only if $x+a$ does, so $\delta$ must be even. In particular, the lowest, and thus best possible value, is $\delta=2$. Vectorial Boolean functions $F \colon \F_{2^n} \rightarrow \F_{2^n}$ satisfying the optimal bound $\delta=2$ are called \emph{Almost Perfect Nonlinear (APN)}. There are different ways to represent (APN) functions on $\F_{2^n}$. We will mention two that are relevant to this paper: Firstly, every function mapping $\F_{2^n}$ to itself can be written (uniquely) as a univariate polynomial of degree at most $2^n-1$. Secondly, if $n=2m$ is even, one can write a function $F \colon \F_{2^m}\times \F_{2^m} \rightarrow \F_{2^m}\times \F_{2^m}$ in a \emph{bivariate} way, i.e.
\begin{equation}
F(x,y)=(F_1(x,y),F_2(x,y)),
\label{eq:bivar}
\end{equation}

where $F_1,F_2 \colon \F_{2^m}\times \F_{2^m} \rightarrow \F_{2^m}$. 

Over the last two decades, several constructions of APN functions have been found, although constructing APN functions "by hand" (i.e. without computer searches) remains a difficult task, and not many infinite families are known. All of the known  theoretic constructions of APN functions can be written as momomials $x\mapsto x^d$ or as so called quadratic functions. 

\begin{definition}
   Let $F \colon \F_{2^n} \rightarrow \F_{2^n}$ be a function written as a polynomial $F(x)=\sum_{i=0}^m a_ix^i$. The \textit{algebraic degree} $\deg(F)$ of $F$ is defined by $\deg(F)=\max_{i \colon a_i \neq 0} \wt(i)$, where $\wt(t)$ denotes the binary weight of the integer $i$, i.e., the number of ones written in the base $2$ representation of $i$.
\end{definition}

A function is called \emph{quadratic} if its algebraic degree is $2$. Similarly, a function written in the bivariate way is quadratic if the functions $F_1,F_2$ from Eq.~\eqref{eq:bivar} only contain the monomials $x^{2^i}y^{2^j}$, $x^{2^i+1}$, $y^{2^j+1}$ for some $i,j \in \N$. This is the case if and only if all of its discrete derivatives $\Delta_a F:=F(x+a)+F(x)$ for $a \in \F_{2^n}^*$ are $\F_2$-linear, see~\cite{Carlet2021}. This also explains why finding quadratic APN functions is comparatively easier than non-quadratic ones: A quadratic function is APN if and only if all its (linear) discrete derivatives  $\Delta_a F=F(x+a)+F(x)$ have a one dimensional kernel over $\F_2$.

Certain equivalence relations leave the differential uniformity and thus the APN property invariant.

Denote by $\Gamma_F = \{(x,F(x)) \colon x \in \F_{2^n}\}$ 
the \textbf{graph} of $F$.
\begin{definition}

	Two functions $F,G \colon \F_{2^n} \rightarrow \F_{2^n}$ are called
	\begin{enumerate}[label=(\roman*)]
		\item \emph{CCZ-equivalent}, if there exists an $\F_2$-affine permutation  
		\[\mathcal{A} : (x,y) \mapsto 
		\begin{pmatrix}
			M & P \\
			N & L
		\end{pmatrix}
		\begin{pmatrix}
			x \\
			y
		\end{pmatrix}
		+
		\begin{pmatrix}
			u \\
			v
		\end{pmatrix}
		\] 
		of $\F_{2^n} \times \F_{2^n}$ such that 
		$\mathcal{A}(\Gamma_F) = \Gamma_G$;

		\item \emph{extended affine equivalent (EA-equivalent)}, \\
			if $F$ and $G$ are CCZ-equivalent with $P = 0$;
		\item \emph{extended linear equivalent (EL-equivalent)}, \\
			if $F$ and $G$ are EA-equivalent with $(u,v) = (0,0)$;
		\item \emph{affine equivalent}, \\
			if $F$ and $G$ are EA-equivalent with $P = N = 0$;
		\item \emph{linear equivalent}, \\
			if $F$ and $G$ are affine equivalent with $(u,v) = (0,0)$.
	\end{enumerate}
\end{definition}
All these relations preserve the APN property, see~\cite{Carlet1998}. EL-equivalence between $F,G$ can be written equivalently as $N(x)+L(F(x))=G(M(x))$ which is readily checked from the definition. A major result of Yoshiara~\cite[Theorem 1]{YoshiaraQuadratic} states that two quadratic APN functions are CCZ-equivalent if and only if they are EA-equivalent. It is then straightforward that under the additional condition 
$F(0)=G(0)=0$ it even suffices to consider EL-equivalence 
(see for instance~\cite[Proposition 2.2.]{kasperszhou}). We summarize these
observations in the following theorem.

\begin{theorem} \label{thm:yoshiara}
Two quadratic APN functions $F,G \colon \F \rightarrow \F$ with $F(0)=G(0)=0$ are CCZ-equivalent if and only if they are EL-equivalent.
\end{theorem}

This paper deals with a special type of  family of quadratic, bivariate  APN functions introduced by G\"olo\u{g}lu~\cite[Theorem III.2]{golouglu2022biprojective}.

\begin{theorem} \label{thm:orig}
	Let $\gcd(3k,m)=1$ and let $\sigma \colon \F_{2^m} \rightarrow \F_{2^m}$ be the field automorphism $x \mapsto x^{2^k}$. Then the  function $F\colon \F_{2^m}\times \F_{2^m} \mapsto \F_{2^m}\times \F_{2^m}$ defined as
	\[F(x,y)=(x^{\sigma+1}+xy^\sigma+y^{\sigma+1},x^{\sigma^2+1}+x^{\sigma^2}y+y^{\sigma^2+1})\]
	is APN.
\end{theorem}

This family was extended in~\cite{calderini}:
\begin{theorem}\label{thm:new}
	Let $\gcd(k,m)=1$ and let $\sigma \colon \F_{2^m} \rightarrow \F_{2^m}$ be the field automorphism $x \mapsto x^{2^k}$. Then the  function $F\colon \F_{2^m}\times \F_{2^m} \mapsto \F_{2^m}\times \F_{2^m}$ defined as
	\[F(x,y)=(x^{\sigma+1}+xy^\sigma+\alpha y^{\sigma+1},x^{\sigma^2+1}+\alpha x^{\sigma^2}y+(1+\alpha^\sigma)xy^{\sigma^2}+\alpha y^{\sigma^2+1})\]
	is APN if $f=X^{\sigma+1}+X+\alpha \in \F_{2^m}[X]$ has no roots in $\F_{2^m}$.
\end{theorem}
Theorem~\ref{thm:new} simplifies to Theorem~\ref{thm:orig} for $\alpha=1$. Indeed, it is easy to see that $f=X^{\sigma+1}+X+1$ has no roots in $\F_{2^m}$ if and only if $\gcd(3k,m)=1$ (see e.g.~\cite[Lemma IV.4.]{golouglu2022biprojective}). \\

%We will from now on refer to the functions in Theorem~\ref{thm:orig} belonging to $\alpha \in \F_{2^m}$ and the field automorphism $\sigma$ as $F_{\alpha,\sigma}$.\\

The obvious questions left open in~\cite{calderini} are:
\begin{enumerate}
	\item If $3\nmid m$, are the functions of the extended family in Theorem~\ref{thm:new} equivalent to the one in Theorem~\ref{thm:orig}? In other words, does the extension yield new inequivalent functions in this case?
	\item How many inequivalent functions does the extended family in Theorem~\ref{thm:new} yield? Note that the corresponding question for the original family was solved in~\cite{inequivalences}.
\end{enumerate}

In Section~\ref{sec:equiv}, we will answer the first question, showing that the new extended family does not yield any new inequivalent functions if $3\nmid m$ (Theorem~\ref{thm:newold}) and in Section~\ref{sec:three} we will answer the second question (Theorem~\ref{thm:count}). In Section~\ref{sec:four}, we tackle some further open questions from~\cite{calderini}: We show that the functions in Theorem~\ref{thm:new} are $3$-to-$1$. This also shows that they have classical Walsh spectrum (see that section for precise definitions). With these results, all major properties of the APN functions in Theorem~\ref{thm:new} are determined.\\

\textbf{NB:} Note that some of the results we prove in this paper were already announced in~\cite{calderini} without proof, citing personal communication with the author. Here, we give the proofs of these results as well as some additional results surrounding them. Maybe most interestingly, the paper illustrates again the powerful approach developed in~\cite{inequivalences} by the author and G\"olo\u{g}lu based on a careful analysis of the automorphism group. 

\section{Determining equivalence} \label{sec:equiv}

We fix for the rest of the paper the following terminology: $K=\F_{2^m}$, $\sigma \colon K\rightarrow K$ is a field automorphism defined by $x\mapsto x^{2^k}$ of order $m$, i.e. $\gcd(k,m)=1$, and $F_{\alpha,\sigma}=(f_{\alpha,\sigma},g_{\alpha,\sigma})$ is the APN function defined on $K^2$ as given by Theorem~\ref{thm:new} with $f_{\alpha,\sigma},g_{\alpha,\sigma} \colon K^2 \rightarrow K$. \\

Crucial objects that play a major role for the APN functions in Theorem~\ref{thm:new} are \emph{biprojective polynomials}, which are polynomials of the form $f(x,y)=a_1x^{\sigma+1}+a_2x^\sigma y+a_3xy^\sigma+a_4y^{\sigma+1} \in K[x,y]$. We refer to~\cite{golouglu2022biprojective,inequivalences} for some information on these polynomials. Not that $f(x,1)$ is a regular \emph{projective polynomial}, which have been then center of a lot of attention for their relations to various combinatorial objects, see~\cite{bluher} for a standard reference. Note that both $f_{\alpha,\sigma},g_{\alpha,\sigma}$ in Theorem~\ref{thm:new} are biprojective polynomials; and the polynomial $X^{\sigma+1}+X+\alpha \in K[X]$ appearing in the condition of Theorem~\ref{thm:new} is a projective polynomial as well.

\subsection{A useful group action on (bi)projective polynomials}
We now define a group action by $G =  K^* \times \GL(2,K)$ on the set of $2\times 2$ matrices $M_{2\times 2}(K)$ as follows: $M=\begin{pmatrix}
	c_1 & c_3 \\
	c_2 & c_4
\end{pmatrix} \in\GL(2,K)$ acts on a matrix $A \in M_{2\times 2}(K)$ for a fixed field automorphism $\sigma$ via \[M\circ A =\begin{pmatrix}
	x & y 
\end{pmatrix} M A (M^{\sigma})^t \begin{pmatrix}
	x^{\sigma} \\
	y^{\sigma}
\end{pmatrix}
\]
where $M^t$ denotes as usual the transpose of $M$ and $M^{\sigma}$ is the matrix where $\sigma$ is applied to every entry. $K^*$ acts on $A\in M_{2\times 2}(K)$ by regular multiplication.

This group action relates to biprojective polynomials in the following way: We can identify $f(x,y)=a_1x^{\sigma+1}+a_2x^\sigma y+a_3xy^\sigma+a_4 y^{\sigma+1}\in K[x,y]$ with $A=\left(\begin{smallmatrix} a_1 & a_3 \\ a_2 & a_4\end{smallmatrix}\right) \in M_{2\times 2}(K)$ and view the group $G$ acting on the set of all biprojective polynomials via this identification. A technical but straightforward calculation then shows (see~\cite{inequivalences}) that using this identification, we get $Mf(x,y) = f(c_1x+c_2y,c_3x+c_4y)$ and $a f(x,y)=a f(x,y)$.  We want to note that this group action is less artificial than it might appear and has also been considered e.g. in~\cite{bbsemifieldchar2, bierbrauer2018family, inequivalences} to deal with combinatorial objects constructed via projective polynomials since it essentially captures linear changes in a (bi)projective polynomial.

The following key lemma shows the intimate connection between the family of APN functions in Theorem~\ref{thm:new} and the group action we just introduced:

\begin{lemma} \label{lem:action}
	Let $\gcd(k,m)=1$. If $(a,M) \in G$ moves $f_{\alpha_1}=x^{\sigma+1}+xy^\sigma+\alpha_1y^{\sigma+1}$ to $f_{\alpha_2}=x^{\sigma+1}+xy^\sigma+\alpha_2y^{\sigma+1}$ then there is a unique $a' \in K^*$ such that $(a',M)$ moves $g_{\alpha_1}=x^{\sigma^2+1}+\alpha_1x^{\sigma^2}y+(1+\alpha_1^{\sigma})xy^{\sigma^2}+\alpha_1y^{\sigma^2+1}$ to $g_{\alpha_2}=x^{\sigma^2+1}+\alpha_2x^{\sigma^2}y+(1+\alpha_2^{\sigma})xy^{\sigma^2}+\alpha_2y^{\sigma^2+1}$.
\end{lemma}
\begin{proof}
	Let $M=\begin{pmatrix}
	c_1 & c_3 \\
	c_2 & c_4
\end{pmatrix} \in\GL(2,K)$. $(a,M) \in G$ moving $f_{\alpha_1}$ to $f_{\alpha_2}$ then implies
	
	\begin{equation}
	(a,M) f_{\alpha_1} =\begin{pmatrix}
	x & y 
\end{pmatrix} M \begin{pmatrix}
	a & a \\
	0 & a\alpha_1
\end{pmatrix} (M^{\sigma})^t \begin{pmatrix}
	x^{\sigma} \\
	y^{\sigma}
\end{pmatrix} = f_{\alpha_2}=\begin{pmatrix}
	x & y 
\end{pmatrix} \begin{pmatrix}
	1 & 1 \\
	0 & \alpha_2
\end{pmatrix} \begin{pmatrix}
	x^{\sigma} \\
	y^{\sigma}
\end{pmatrix}.
	\label{eq:f}
	\end{equation}
This implies $M \begin{pmatrix}
	a & a \\
	0 & a\alpha_1
\end{pmatrix} (M^{\sigma})^t=\begin{pmatrix}
	1 & 1 \\
	0 & \alpha_2
\end{pmatrix}$.

In particular, by taking determinants, $\det(M)^{\sigma+1}a^2\alpha_1=\alpha_2$ and for each $M$ there is exactly one value for $a$ that is admissible.

Similarly, $(a',M)$ moves $g_{\alpha_1}$ to $g_{\alpha_2}$ in the following way
	\begin{equation}
	(a',M)g_{\alpha_1} =\begin{pmatrix}
	x & y 
\end{pmatrix} M \begin{pmatrix}
	a' & a'(1+\alpha_1^\sigma) \\
	a'\alpha_1 & a'\alpha_1
\end{pmatrix} (M^{\sigma^2})^t \begin{pmatrix}
	x^{\sigma^2} \\
	y^{\sigma^2}
\end{pmatrix} = g_{\alpha_2}=\begin{pmatrix}
	x & y 
\end{pmatrix} \begin{pmatrix}
	1 & 1+\alpha_2^\sigma \\
	\alpha_2 & \alpha_2
\end{pmatrix} \begin{pmatrix}
	x^{\sigma^2} \\
	y^{\sigma^2}
\end{pmatrix}.
	\label{eq:g}
	\end{equation}
Again, taking determinants yields $a'^2\det(M)^{\sigma^2+1}\alpha_1^{\sigma+1}=\alpha_2^{\sigma+1}$ and $M$ determines $a'$ uniquely. Clearly, scaling $M$ by a non-zero constant just means one has to adjust $a,a'$, so we  can assume without loss of generality $a=a'=\det(M)=1$, i.e. just consider $M f_{\alpha_1}=f_{\alpha_2}$ and $M g_{\alpha_1}=g_{\alpha_2}$.

In this case, $((M^{\sigma})^t)^{-1}=\begin{pmatrix}
	c_4^\sigma &c_2^\sigma \\
	c_3^\sigma & c_1^\sigma
\end{pmatrix}$ and we can rewrite the condition  from Eq.~\eqref{eq:f} as
\begin{equation*} 
\begin{pmatrix}
	c_1 & c_3 \\
	c_2 & c_4
\end{pmatrix}\begin{pmatrix}
	1 & 1 \\
	0 & \alpha_1
\end{pmatrix} = \begin{pmatrix}
	1 & 1 \\
	0 & \alpha_2
\end{pmatrix}\begin{pmatrix}
	c_4^\sigma &c_2^\sigma \\
	c_3^\sigma& c_1^\sigma
\end{pmatrix}.
\end{equation*}
This leads to four equations in the entries: 
\begin{align}
	c_1 &= c_3^\sigma+c_4^\sigma\label{eq:c1}  \\
	c_1+\alpha_1c_3 &= c_2^\sigma+c_1^\sigma \\
	c_2  &= \alpha_2 c_3^\sigma \label{eq:c2} \\
	c_2+\alpha_1c_4 &= \alpha_2 c_1^\sigma.\label{eq:c4}
\end{align}
Eliminating $c_1,c_2$ via Eqs.~\eqref{eq:c1},~\eqref{eq:c2} yields
\begin{align}
	c_3^\sigma+c_4^\sigma+\alpha_1c_3 &= (1+\alpha_2)^\sigma c_3^{\sigma^2}+c_4^{\sigma^2} \label{eq:c5}\\
	\alpha_2 c_3^\sigma+\alpha_1c_4 &= \alpha_2 (c_3^{\sigma^2}+c_4^{\sigma^2}).\label{eq:c6}
\end{align}

Let us now consider $M g_{\alpha_1}=g_{\alpha_2}$, see Eq.~\eqref{eq:g}, leading to
\[\begin{pmatrix}
	c_1 & c_3 \\
	c_2 & c_4
\end{pmatrix}\begin{pmatrix}
	1 & 1+\alpha_1^\sigma \\
	\alpha_1 & \alpha_1
\end{pmatrix} = \begin{pmatrix}
	1 & 1+\alpha_2^\sigma \\
	\alpha_2 & \alpha_2
\end{pmatrix}\begin{pmatrix}
	c_4^{\sigma^2} &c_2^{\sigma^2} \\
	c_3^{\sigma^2} & c_1^{\sigma^2}
\end{pmatrix}.\]
These again yields four equations:
\begin{align*}
	c_1 +\alpha_1c_3&= c_4^{\sigma^2}+(1+\alpha_2^\sigma)c_3^{\sigma^2}  \\
	c_1(1+\alpha_1^\sigma)+\alpha_1c_3 &= c_2^{\sigma^2}+(1+\alpha_2^\sigma)c_1^{\sigma^2}\\
	c_2  +\alpha_1c_4&= \alpha_2 (c_3^{\sigma^2}+c_4^{\sigma^2})  \\
	c_2(1+\alpha_1^\sigma)+\alpha_1c_4 &= \alpha_2 (c_1^{\sigma^2}+c_2^{\sigma^2}).
\end{align*}
It remains to show that Eqs.~\eqref{eq:c1} to~\eqref{eq:c4} imply the four equations above. We again eliminate $c_1,c_2$ via Eqs.~\eqref{eq:c1},~\eqref{eq:c2}, resulting in
\begin{align*}
	c_3^\sigma+c_4^\sigma+\alpha_1c_3&= c_4^{\sigma^2}+(1+\alpha_2^\sigma)c_3^{\sigma^2}  \\
	(c_3^\sigma+c_4^\sigma)(1+\alpha_1^\sigma)+\alpha_1c_3 &= \alpha_2^{\sigma^2}c_3^{\sigma^3}+(1+\alpha_2^\sigma)(c_3^{\sigma^3}+c_4^{\sigma^3}) \\
	\alpha_2c_3^\sigma  +\alpha_1c_4&= \alpha_2 (c_3^{\sigma^2}+c_4^{\sigma^2})  \\
	(\alpha_2c_3^\sigma)(1+\alpha_1^\sigma)+\alpha_1c_4 &= \alpha_2 (c_3^{\sigma^3}+c_4^{\sigma^3}+\alpha_2^{\sigma^2}c_3^{\sigma^3}).
\end{align*}
It is easy to verify that the first equation is Eq.~\eqref{eq:c5}; the second equation is Eq.~\eqref{eq:c5} plus Eq.~\eqref{eq:c5} after applying $\sigma$,  plus Eq.~\eqref{eq:c6} after applying $\sigma$; the third equation is Eq.~\eqref{eq:c6} and the fourth equation is $\alpha_2$ times Eq.~\eqref{eq:c5} after applying $\sigma$. This proves the claim. 
\end{proof}

Lemma~\ref{lem:action} shows that if $f_{\alpha_1}$ and $f_{\alpha_2}$ are in the same orbit under the action of $G$ then $g_{\alpha_1}$ and $g_{\alpha_2}$ are also in the same orbit. If this happens, there are $a,a'\in K^*$ and $c_1,c_2,c_3,c_4\in K$ with $c_1c_4+c_2c_3 \neq 0$ such that $F_{\alpha_1}(c_1x+c_2y,c_3x+c_4y)=(af_{\alpha_2}(x,y),a'g_{\alpha_2}(x,y))$. In particular, in this case $F_{\alpha_1}$ and $F_{\alpha_2}$ are linear equivalent. We summarize:

\begin{proposition} \label{prop:moving}
	Let $F_{\alpha_1,\sigma}=(f_{\alpha_1,\sigma}, g_{\alpha_1,\sigma}),F_{\alpha_2,\sigma}=(f_{\alpha_2,\sigma}, g_{\alpha_2,\sigma})$ be two APN functions on $K^2$ defined as in Theorem~\ref{thm:new}. If $f_{\alpha_1,\sigma}$ and $f_{\alpha_2,\sigma}$ are in the same orbit under $G$ then $F_{\alpha_1,\sigma}$ and $F_{\alpha_2,\sigma}$ are linear equivalent.
\end{proposition}
This means that we can determine the (in)equivalence of these APN functions by working with the group action of $G$ on the set of biprojective polynomials. 
The orbits and stabilizer sizes of this action  relevant to our case were determined in~\cite[Lemma 7]{inequivalences}:
\begin{proposition} \label{prop:lemma7}
		$G$ acts transitively on the set of biprojective polynomials $f(x,y)=a_1x^{\sigma+1}+a_2x^\sigma y+a_3xy^\sigma+a_4y^{\sigma+1}$ with $a_1 \neq 0$ such that $f(x,1)$ has no roots in $K$. In other words, all such polynomials are in the same orbit under $G$. The size of the stabilizer of any polynomial in this set is $3(2^m-1)$. More precisely, for each $c\in K^*$ there are exactly three elements $(a,M)$ in the stabilizer such that $\det(M)=c$. 
\end{proposition}

We can thus combine Propositions~\ref{prop:moving} and~\ref{prop:lemma7} to the following result.
\begin{theorem}\label{thm:newold}
	Let $K=\F_{2^m}$, $\sigma \colon x \mapsto x^{2^k}$ with $\gcd(k,m)=1$. Any two APN functions from Theorem~\ref{thm:new}, $F_{\alpha_1,\sigma},F_{\alpha_2,\sigma}$, are linear equivalent. In particular, if $3\nmid m$, then all APN functions in Theorem~\ref{thm:new} are equivalent to functions in G\"olo\u{g}lu's original family given in Theorem~\ref{thm:orig}.
\end{theorem}
\begin{proof}
	Follows immediately from Proposition~\ref{prop:moving} and the transitivity of the group action as proven in Proposition~\ref{prop:lemma7}.
\end{proof}

\section{Counting the number of inequivalent APN functions in the extended family}\label{sec:three}

In this section, we use group theoretical tools to determine the number of inequivalent function in the extended APN family. The machinery is a variant used by the author and G\"olo\u{g}lu in a series of papers to establish (in)equivalence of combinatoral structures~\cite{inequivalences, golouglu2022exponential, semifield2}.\\

We start by spotting a simple equivalence:

\begin{proposition} \label{prop:inverse}
	Let $\sigma \colon x\mapsto x^{2^k}$ and $\overline{\sigma} \colon x\mapsto x^{2^{m-k}}$ be field automorphisms on $K=\F_{2^m}$ that are inverse to each other. Then $F_{\alpha,\sigma}$ and $F_{\alpha^{\overline{\sigma}},\overline{\sigma}}$ are linear equivalent.
\end{proposition}
\begin{proof}
	Recall $F_{\alpha,\sigma}=(f_{\alpha,\sigma},g_{\alpha,\sigma})$. We apply $\overline{\sigma}$ to $f_{\alpha,\sigma}$ and $\overline{\sigma}^2$ to $g_{\alpha,\sigma}$, we get
	\[F'=(x^{\overline{\sigma}+1}+x^{\overline{\sigma}}y+\alpha^{\overline{\sigma}}y^{\overline{\sigma}+1},x^{\overline{\sigma}^2+1}+\alpha^{\overline{\sigma}^2}xy^{\overline{\sigma}^2}+(1+\alpha^{\overline{\sigma}^2})x^{\overline{\sigma}^2}y+\alpha^{\overline{\sigma}^2}y^{\overline{\sigma}^2+1}).\]
	Now perform a shift $x\mapsto x+y$ and we get 
		\[F''=(x^{\overline{\sigma}+1}+xy^{\overline{\sigma}}+\alpha^{\overline{\sigma}}y^{\overline{\sigma}+1},x^{\overline{\sigma}^2+1}+\alpha^{\overline{\sigma}^2}x^{\overline{\sigma}^2}y+(1+\alpha^{\overline{\sigma}^2})xy^{\overline{\sigma}^2}+\alpha^{\overline{\sigma}}y^{\overline{\sigma}^2+1}).\]
		This is exactly $F_{\alpha^{\overline{\sigma}},\overline{\sigma}}$.
\end{proof}
Note that $p_1=X^{\sigma+1}+X+\alpha$ has no roots in $K$ if and only if $p_2=X^{\overline{\sigma}+1}+X+\alpha^{\sigma}$ has no roots in $K$. Indeed, applying $\overline{\sigma}$ to $p_1$ and then a transformation $X\mapsto X+1$ exactly yields $p_2$.

Theorem 3 in~\cite{inequivalences} gives a very strong and general tool to determine when two biprojective APN functions are equivalent or not. More precisely, it gives very strong conditions on the form of potential equivalences between biprojective functions. The result relies on group theoretic properties based on the very special automorphism groups of such functions. We will not explain the entire theory here, instead just focus on the parts needed to our case. The interested reader is invited to read the extended exposition in~\cite{inequivalences}. Still, we will need to introduce a bit of notation to explain this result.\\

We define the group of EL-mappings (i.e., the set of mappings that correspond to extended linear transformations on graphs) as 
\begin{equation*}
\ELM = \Bigg\{ \begin{pmatrix}
	M & 0 \\
	N & L
\end{pmatrix} \in \GL(K^2)\cong \GL(2m,\F_2)\Bigg\}.
\end{equation*}
where $M,N,L$ are linear mappings on $K$. 
%\faruk{Here $\F$, $\F \times \F$  and $\M \times \M$ are confusing.}

Further denote by 
\[\Aut_{\EL}(F)=\{\mathcal{A} \in \ELM \colon \mathcal{A}(\Gamma_F)=\Gamma_F\}\]
 the group of EL-automorphisms of a function $F$. Clearly, if $F$ and $G$ are EL-equivalent, the corresponding EL-automorphism groups are conjugate in $\ELM$; this is essentially just a simple application of the orbit-stabilizer theorem.
\begin{proposition} \label{prop:conjugated} {\cite[Proposition 1]{inequivalences}}
	Assume $F,G \colon \F \rightarrow \F$ are EL-equivalent via the EL-mapping $\gamma \in \ELM$, i.e., $\gamma(\Gamma_F)=\Gamma_G$. Then $\Aut_{\EL}(F)=\gamma^{-1} \Aut_{\EL}(G) \gamma$.
\end{proposition}

The important fact is that all biprojective functions have a big subgroup in the group of EL-automorphisms that can be written down in a simple and explicit way, this was the main fact leveraged in~\cite{inequivalences}. Since the functions $F_{\alpha,\sigma}$ we investigate are a special case of biprojective function, we thus have the following.

\begin{proposition}{\cite[Lemma 3]{inequivalences}}
	Let $M_a=\left(\begin{smallmatrix} a & 0 \\ 0 & a \end{smallmatrix}\right) \in \GL(2,K)$, $L_a=\left(\begin{smallmatrix} a^{\sigma+1} & 0 \\ 0 & a^{\sigma^2+1} \end{smallmatrix}\right)\in \GL(2,K)$ for  $a \in K^*$. Then 
	\[Z^\sigma=\left\{\begin{pmatrix}
	M_a & 0 \\
	0 & L_a
\end{pmatrix} \colon a \in K^*\right\} \leq \Aut_{\EL}(F_{\alpha,\sigma})\]
for any admissible $\alpha$.
\end{proposition}
Clearly, $Z^\sigma$ is a cyclic group of order $|K^*|=2^m-1$.

Let $p$ be a Zsygmondy prime of $2^m-1$, i.e. a prime number $p$ such that $p|2^m-1$ but $p\nmid 2^r-1$ for any $r<m$. Such a $p$ always exists by a theorem of Zsygmondy if $m>1$, $m\neq 6$, see~\cite[Chapter IX., Theorem 8.3.]{huppert}. Since $Z^\sigma$ is cyclic and $p$ divides $|Z^\sigma|=2^m-1$, we have a unique Sylow $p$-group in $Z^\sigma$, which we will denote by $Z^\sigma_p$. If $R$ is the unique Sylow $p$-subgroup of $K^*$ then clearly 
\[Z^\sigma_p=\left\{\begin{pmatrix}
	M_a & 0 \\
	0 & L_a
\end{pmatrix} \colon a \in R\right\} \leq \Aut_{\EL}(F_{\alpha,\sigma}).\]

We also denote by $C_{\alpha,\sigma}$ the centralizer of $Z^\sigma$ in $\Aut_{\EL}(F_{\alpha,\sigma})$.  

With this notation in place, we may state the equivalence result from~\cite{inequivalences}. 
Applied to the functions $F_{\alpha,\sigma}$ from Theorem~\ref{thm:new}, it states the following:

\begin{theorem} \label{thm:projequiv} {\cite[Theorem 3]{inequivalences}}
Assume $m>2$ and $m \notin \{4,6\}$. Let $F_{\alpha_1,\sigma}=(f_1,g_1)$ and $F_{\alpha_2,\tau}=(f_2,g_2)$ be APN functions from Theorem~\ref{thm:new} defined on $K^2=\F_{2^m}^2$ with field automorphisms $\sigma,\tau$.
Assume further that
	\begin{equation}
		p \text{ does not divide } |C_{\alpha,\sigma}|/(2^m-1).
	\tag{C}\label{eq:condition}
	\end{equation}
	Then $F_{\alpha_1,\sigma},F_{\alpha_2,\tau}$ cannot be EL-equivalent unless $\sigma=\tau$ or $\sigma=\tau^{-1}$. 
\end{theorem}

We thus only have to verify Condition~\eqref{eq:condition} for all $F_{\alpha,\sigma}$. Similar calculations are done in~\cite{inequivalences}, where it is shown that if $\gamma \in C_{\alpha,\sigma}$ then $\gamma=\left(\begin{smallmatrix}
	M & 0 \\
	0 & L
\end{smallmatrix}\right)$ where $M,L \in \GL(2,K)$ and $L$ is a diagonal matrix, this is shown in~\cite[Proof of Lemma 8]{inequivalences} and essentially hinges on the fact that the centralizer of $M_a=\{\left(\begin{smallmatrix} a & 0 \\ 0 & a \end{smallmatrix}\right) \colon a \in R\}\leq \GL(2,K)$ in the group $\GL(K^2)\cong \GL(2m,\F_2)$ is precisely $\GL(2,K)$, which was shown in~\cite[Lemma 5.7.]{golouglu2022exponential}. 

\begin{lemma}
	All elements in $C_{\alpha,\sigma}$ are of the form
	\[\begin{pmatrix}
		M & 0\\
		0 & L
	\end{pmatrix},\]
	where $M,L \in \GL(2,K) \leq GL(K^2)$ and $L$ is a diagonal matrix. For each fixed value of $c \in K^*$ there are exactly $3$ different elements in $C_{\alpha,\sigma}$ such that $\det(M)=c$.
	
	In particular, $|C_{\alpha,\sigma}|=3(2^m-1)$ and Condition~\eqref{eq:condition} is satisfied if $m>2$.
\end{lemma}
\begin{proof}
	From the discussion above, it is already clear that all elements in $C_{\alpha,\sigma}$ are of the form mentioned above. It remains to show that for each fixed value of $c \in K^*$ there are exactly $3$ different elements in $C_{\alpha,\sigma}$ such that $\det(M)=c$. \\
	So let us assume we have $\left(\begin{smallmatrix}
		M & 0\\
		0 & L
	\end{smallmatrix} \right) \in C_{\alpha,\sigma}.$ We then have $L\circ F_{\alpha,\sigma}=F_{\alpha,\sigma}\circ M$ Set $M=\left(\begin{smallmatrix}
		c_1 & c_2\\
		c_3 & c_4
	\end{smallmatrix} \right)$ and $L=\left(\begin{smallmatrix}
		a & 0\\
		0& a'
	\end{smallmatrix} \right)$. Then $L\circ F_{\alpha,\sigma}=F_{\alpha,\sigma}\circ M$ is equivalent to
	\begin{align*}
		a f_{\alpha,\sigma}(x,y)&=f_{\alpha,\sigma}(c_1x+c_2y,c_3x+c_4y) \\
			a' g_{\alpha,\sigma}(x,y)&=g_{\alpha,\sigma}(c_1x+c_2y,c_3x+c_4y). 	
	\end{align*}
	By Lemma~\ref{lem:action}, if the first equation is satisfied, there is precisely one $a'$ that satisfied the second equation. We thus only have to check the first equation, which is equivalent to $(a,M)$ being in the stabilizer of the group action defined by $G$ in Section~\ref{sec:equiv}. The result then follows immediately with Proposition~\ref{prop:lemma7}. \\
	Note that the Zsygmondy prime $p$ does not divide $2^r-1$ for $r<m$, in particular it does not divide $3=2^2-1$ if $m>2$, so Condition~\eqref{eq:condition} is satisfied.
	
\end{proof}

Having proven Condition~\eqref{eq:condition}, Theorem~\ref{thm:projequiv} together with Proposition~\ref{prop:inverse} and Theorem~\ref{thm:newold} yield the main result of this section.
\begin{theorem} \label{thm:count}
	 Let $m>2$ and $m \notin \{4,6\}$. Let further $F_{\alpha_1,\sigma}$ and $F_{\alpha_2,\tau}$ be two APN functions from Theorem~\ref{thm:new} defined on $K^2=\F_{2^m}^2$ with field automorphisms $\sigma,\tau$. The two functions are CCZ-equivalent if and only if $\sigma=\tau$ or $\sigma=\overline{\tau}$, where $\overline{\tau}$ is the inverse of $\tau$ is the automorphism group of $K$. \\
	There are in total $\varphi(m)/2$ CCZ-inequivalent functions in the family from Theorem~\ref{thm:new} defined on $K^2$.
\end{theorem}
Note that the result also holds in the $m=2$ and $m=4$ case because in both cases, the only admissible field automorphisms are $\sigma \colon x \mapsto x^2$ and (in the $m=4$ case) its inverse, so Theorem~\ref{thm:newold} and Proposition~\ref{prop:inverse} alone already yield the result. For the $m=6$ case, we have checked by computer that the result also holds. We want to emphasize that this means that the choice of $\alpha$ in Theorem~\ref{thm:new} does not matter at all in the sense that any choice of $\alpha$ leads to equivalent functions. The original family by G\"olo\u{g}lu just chose $\alpha=1$, the only "gap" was that  $\alpha=1$ does not satisfy the condition that $X^{\sigma+1}+X+\alpha$ must not have any roots in $\F_{2^m}$ if 3 divides $m$.  

\section{The image sets and Walsh spectra of the APN functions}\label{sec:four}

One of the most important properties of APN functions is their Walsh spectrum, which captures how a resistant a function is towards linear attacks when used as a function in a substitution permutation network, we again refer the reader to~\cite{Carlet2021} for details. 
\begin{definition}
	Let $F \colon \F_{2^n} \rightarrow \F_{2^n}$ be a mapping. We define
	\[W_F(b,a) = \sum_{x \in \F_{2^n}}(-1)^{\Tr(bF(x)+ax)} \in \Z\]
	for all $a,b \in \F$. We call the multisets
	\[  \{*W_F(b,a) \colon b \in \F_{2^n}^*, a\in \F_{2^n}*\}  \text{ and }
	   \{*|W_F(b,a)|\colon b \in \F_{2^n}^*, a\in \F_{2^n}*\} \]
	the \emph{Walsh spectrum} and the \emph{extended Walsh spectrum} of $F$, respectively.
\end{definition} 

The extended Walsh spectrum is invariant under CCZ-equivalence. Most known APN functions in even dimension $n$ have the so called \emph{classical} (or Gold-like) extended Walsh spectrum, which contains the values $0,2^{n/2},2^{(n+2)/2}$ precisely $(2^n-1)2^{n-2}$ times, $(2^n-1)2^{n+1}/3$ times and $(2^n-1)2^n/3$ times, respectively. 

Recently, the following simple criterion to determine the Walsh spectrum of an APN function from its value distributions was found.

\begin{theorem}{\cite[Theorem 1]{kkk}} \label{prop:walsh}
	Let $n$ be even and $F \colon \F_{2^n} \rightarrow \F_{2^n}$ be a quadratic APN function such that 
	\begin{itemize}
		\item $F(0)=0$, and
		\item Every $y \in \image(F)\setminus \{0\}$ has at least $3$ preimages. 
	\end{itemize}
	Then $F(x)=0$ if and only $x=0$ and every $y \in \image(F)\setminus \{0\}$ has precisely $3$ preimages (i.e., $F$ is $3$-to-$1$). Additionally, $F$ has classical Walsh spectrum.
\end{theorem}
Note that APN $3$-to-$1$ functions are the functions with the smallest possible image sets for APN functions in even dimension~\cite{kkk}.

\begin{theorem} \label{thm:31}
	$F_{\alpha,\sigma}$ is $3$-to-$1$ and has classical Walsh spectrum for all admissible values of $\alpha,\sigma$.
\end{theorem}
\begin{proof}
	We apply Theorem~\ref{prop:walsh}. It is immediate that $F_{\alpha,\sigma}(0,0)=(0,0)$. 
	Let us consider the equation $F_{\alpha,\sigma}\circ M = F_{\alpha,\sigma}$ for $M \in \GL(2,K)\leq \GL(K^2)$. This is equivalent to $(1,M)$ being both in the stabilizer of $f_{\alpha,\sigma}$ and $g_{\alpha,\sigma}$ where we use again the group action defined in Section~\ref{sec:equiv}. The matrix equations that characterize $(1,M)$ being in the stabilizers are written down in Eqs.~\eqref{eq:f} and~\eqref{eq:g} and, taking again determinants, lead to $\det(M)^{\sigma+1}=\det(M)^{\sigma^2+1}=1$, which is equivalent to $\det(M)=1$. 
	By Proposition~\ref{prop:lemma7}, there are thus 3 distinct $M$ (one of them being the identity matrix $I$) such that $(1,M)$ is in the stabilizer of $f_{\alpha,\sigma}$. For these 3 matrices, we have $F_{\alpha,\sigma}(M(x,y))= F_{\alpha,\sigma}(x,y)$ for all $x,y \in K$.  Let us say the three matrices with this property are $I,M,N$. Clearly, if $M$ has this property, then so does $M^i$ for any $i \in \N$, this immediately implies that $N=M^2$ and $M^3=I$. 
	
	To show that $F_{\alpha,\sigma}$ is $3$-to-$1$, it is enough to show that $M$ has only $(0,0)$ as a fixed point, i.e. $M(x,y)\neq (x,y)$ for all $(0,0)\neq (x,y) \in K^2$, which is equivalent to $M$ not having $1$ as an eigenvalue. Since $M^3=I$ the minimal polynomial of $M$ has to be a divisor of $X^3+1=(X+1)(X^2+X+1)$ but not $X+1$ since $M\neq I$. Note that $X^2+X+1$ has the roots $\omega \in \F_4\setminus \F_2$, so it is an irreducible polynomial if $m$ is odd. In this case, the minimal polynomial has to be $X^2+X+1$ and $1$ is not an eigenvalue. If $m$ is even, then $X^3+1=(X+1)(X+\omega_1)(X+\omega_2)$ where $\omega_1,\omega_2\in\F_4\setminus \F_2$. Assume $M=\left(\begin{smallmatrix} a & b \\c & d\end{smallmatrix}\right)$ has eigenvalue $1$, so the characteristic polynomial is $\chi_M(X)=(X+1)(X+\omega)=X^2+(\omega+1)X+\omega$ for some $\omega \in \F_4\setminus \F_2$. On the other hand $\chi_M(X)=\det\left(\begin{smallmatrix} X+a & b \\c & X+d\end{smallmatrix}\right)=X^2+(a+d)X+ad+bc$, so comparing the constant coefficent we have $\omega=ad+bc$. This contradicts $ad+bc=\det(M)=1$. We conclude that $1$ is not an eigenvalue of $M$, so $M(x,y)\neq(x,y)$ for any $(x,y)\neq(0,0)$. 
	
	So $F_{\alpha,\sigma}(x,y)=F_{\alpha,\sigma}(M(x,y))=F_{\alpha,\sigma}(M^2(x,y))$ for any $(x,y) \in K^2$, where $(x,y)$, $M(x,y)$, $M^2(x,y)$ are pairwise distinct if $(x,y)\neq (0,0)$. So all conditions of Theorem~\ref{prop:walsh} are satisfied and the result follows.
\end{proof}
This is particularly interesting since $F_{\alpha,\sigma}$ is $3$-$1$ for \emph{all} admissible $\alpha$. Recall that by Theorem~\ref{thm:newold} different $\alpha$ (for the same $\sigma$) yield CCZ-equivalent functions. This means that in the equivalence CCZ equivalence class of $F_{\alpha,\sigma}$ there are many $3$-to-$1$ functions. More precisely, there are at least $\frac{2^{m+1}-2}{6}$ or $\frac{2^{m+1}+2}{6}$ (depending on $m$ even or odd) such functions in the CCZ equvialence class of $F_{\alpha,\sigma}$ since this is the number of polynomials $p=X^{\sigma+1}+X+\alpha \in K[X]$ with no roots in $K$, see~\cite[Theorem 5.6.]{bluher}.

\section{Conclusion and open problems}
In this note we proved that the recent extension of an APN family of G\"olo\u{g}lu only yields new APN functions on $\F_{2^{2m}}$ if $3| m$ (Theorem~\ref{thm:newold}). In this case, we counted the number of inequivalent functions the extension yields (Theorem~\ref{thm:count}) and we showed that all functions in the family are $3$-to-$1$ and have classical Walsh spectrum (Theorem~\ref{thm:31}). These results hinge on group theoretic tools developed in general for biprojective functions in~\cite{inequivalences}, based on large cyclic subgroups in the automorphism groups of these functions. We state some interesting open problems:

\begin{enumerate}
	\item Is it possible to generalize the group theoretic tools from~\cite{inequivalences} from biprojective functions to a wider class of functions? Note that a similar approach had already been used for power functions~\cite{DempwolffPower,YoshiaraPower} before.
	\item Give a lower bound on the number of APN functions that is better than the current bound given in~\cite{kasperszhou}. 
\end{enumerate}

Based on a similar conjecture for the (on a theoretical level) related combinatorial objects of semifields by Kantor~\cite{kantor2006finite}, in particular the explicit construction of semifields in~\cite{Kantor03}, we conjecture the following:
\begin{conjecture}
	Let $N(Q)$ be the number of APN functions on $\F_Q$. $N(Q)$ is not bounded from above by a polynomial.  
\end{conjecture}
Note that the best current bound by Kaspers and Zhou~\cite{kasperszhou} are not even linear in $Q$ and thus quite far away from this bound. Better constructions or non-constructive arguments (for instance, using probabilistic methods) thus seem to be needed to tackle this conjecture.

\bibliographystyle{acm}

\bibliography{references}

\end{document}